\documentclass{amsart}

\usepackage{amssymb}
\usepackage{amsfonts}
\usepackage{verbatim}
\usepackage{pstricks}

\newtheorem{thm}{Theorem}
\newtheorem{rmk}[thm]{Remark}
\newtheorem{cor}[thm]{Corollary}

\newcommand{\la}{\langle}
\newcommand{\ra}{\rangle}

\title{Differential forms, fluids, and finite models}
\author{Scott O. Wilson}
\date{January 13, 2011}

\begin{document}


\begin{abstract}
By rewriting the Navier-Stokes equation in terms of differential forms we give a formulation 
which is abstracted and reproduced in a finite dimensional setting. We give two examples of these finite models
and, in the latter case, prove some approximation results. Some useful properties of these finite models are derived. 
\end{abstract}

\maketitle


\section{Introduction}

Increasingly the lines between algebraic topology, applied mathematics, differential geometry, analysis, mathematical physics, etc.,  are becoming difficult to even imagine, as tools in one area become useful or essential in another.

Tracing the history of one famous line of research, the study of the heat equation, one sees such an interplay.
For example, the abstraction of the Euclidean analysis problem to the Riemannian setting led to spectral invariants and new notions of cohomology, while the further abstraction to an algebraic setting led to heat flow on graphs and simplicial complexes whose invariants include combinatorial counts such as spanning trees and their generalizations. 
At the same time, approximation theory was used to better understand mathematical invariants such as torsion, and 
conversely, techniques from algebra, topology, and analysis were (and still are) being used to construct models of heat flow.

Fundamental in this story is some abstract formulation of the problem in question, so that tools in one area
can be used to answer and pose questions in another.

The purpose of this paper is to give one such abstraction for the Navier-Stokes equation, suitable for expression in both a  smooth Riemannian manifold and finite cochain complex setting. The results we obtain suggest a great potential for these two viewpoints to inform one another (e.g. Corollary~\ref{cor:steady}). 
Additionally, we indicate how this viewpoint may be used to simulate fluids, and expect it will lead to interesting 
dynamical and combinatorial problems in these finite models. 

I would like to thank Jozef Dodziuk and Dennis Sullivan for helpful discussions related to this work.

\section{Viscous incompressible fluids}

In this section we explain how to express the Navier-Stokes equation in the language of differential forms on Riemannian manifold. The formulation obtained in the zero viscosity case, which corresponds to Euler's equation, agrees with \cite{AMR} (p. 588).
We then give some basic properties which follow immediately from algebraic considerations.

The Navier-Stokes equation for a divergence free vector field representing an incompressible
homogeneous fluid of viscosity $\nu \geq 0$ is given by
\begin{align}
\frac{\partial u}{\partial t} + ( u \cdot \nabla ) u & = - \nabla p + \nu \Delta u \label{ns} \\
\mathrm{div} \, u &= 0 \nonumber
\end{align}
where $u$ is a time dependent vector field in $\mathbb{R}^3$ and the pressure, $p$, is a time
dependent function on $\mathbb{R}^3$. 
The latter equation, that the vector  field $u$ is divergence free, is referred to as the incompressibility condition. 
Using the vector identity
\[
( u \cdot \nabla ) u = - u \times \mathrm{curl} \, u + \frac{1}{2} \nabla \| u \|^2 
\]
we may re-write equation \eqref{ns} as
\begin{align}
\frac{\partial u}{\partial t}  & =  - u \times \mathrm{curl} \, u + \frac{1}{2} \nabla \| u \|^2  - \nabla p + \nu \Delta u \label{ns2}
\end{align}

We now rewrite this equation in terms  of a differential 1-form on $\mathbb{R}^3$, thought of as a Riemannian manifold
with the standard Euclidean metric, where 1-forms and vector fields are identified by the metric.
Under this identification, the cross product is given by $(\omega , \eta) \to \star( \omega \wedge \eta)$, and the curl operator is $\star d$, where $\star$ is the Hodge-star operator. 
Also, the divergence free vector fields correspond to those 1-forms $\omega$ which satisfy $d^* \omega = 0$, and the Laplacian of a divergence free 1-form is given by $-d^*d \omega$.
Putting this together we can rewrite equation \eqref{ns2}, the Navier-Stokes on the Riemannian manifold $\mathbb{R}^3$, as
\begin{align}
\frac{ \partial \omega}{ \partial t} &=  - \star( \omega \wedge \star d \omega) + \frac{1}{2} d \| \omega \|^2 - d p - \nu d^* d \omega \label{ns3} \quad \quad d^* \omega = 0
\end{align}
for a time dependent 1-form $\omega$. For any time dependent 1-form $\omega(t)$ on a Riemannian manifold, we refer to this system as the \emph{Navier-Stokes equation in a Riemannian manifold}.
In the zero-viscosity case, $\nu = 0$, the 
equation is referred to as \emph{Euler's equation on a Riemannian manifold}.

Assuming our manifold is compact\footnote{One may instead use differential forms with compact support, or those with rapid decay, as in the original Helmholtz decomposition.}, we now make one more simplification. By the Hodge decomposition of differential forms we see that, since the left hand side of equation \eqref{ns3} is in the kernel of $d^*$, the $d$-exact part of the right hand side must be zero. Therefore the Navier-Stokes equation in a Riemannian manifold is equivalent to 
\begin{align}
\frac{ \partial \omega}{ \partial t} &=  \pi \left(- \star( \omega \wedge \star d \omega) \right)
- \nu d^* d \omega \quad \quad 
d^* \omega = 0 \label{ns4}
\end{align}
where $\pi$ denotes the orthogonal projection onto the kernel of $d^*$, i.e. the co-closed $1$-forms.

Setting $\nu = 0$ we obtain the analogous Euler equation in a Riemannian manifold
\begin{align}
\frac{ \partial \omega}{ \partial t} &=  \pi \left( -\star( \omega \wedge \star d \omega) \right)
  \quad \quad 
d^* \omega = 0 \label{eu1}
\end{align}

From this formulation, several properties are apparent. 
First, if additionally $\omega_0$ is $d$-closed (so $\omega_0$ is harmonic) then 
the right hand side of \eqref{ns4} is zero, and therefore $\omega(t) = \omega_0$ is a \emph{steady state solution}.
Non-harmonic steady state solutions to Euler's equation are known to exist in general.

Second, if $\omega(t)$ is a solution to 
$\eqref{ns4}$, then by taking the inner product of both sides of $\eqref{ns4}$ with
$\omega(t)$ we obtain
\begin{align}
\frac{ \partial}{ \partial t} \| \omega \|^2 = -2 \nu \|d \omega \|^2 \label{norm}
\end{align}
Here we have used that $d^*$ is the adjoint of $d$ and that 
\begin{align}
\la \pi \left( - \star( \omega \wedge \star d \omega) \right) , \omega \ra = 
\la - \star( \omega \wedge \star d \omega) , \omega \ra =
\la d \omega , \omega \wedge \omega \ra = 0
\end{align}
In particular, we see that an Eulerian flow is norm preserving, while for a general Navier-Stokes flow
the  norm is non-increasing, as a function of the vorticity, $d\omega$.

\section{Algebraic and analytic considerations}

The equations of the previous section are well defined on the space of smooth differential
forms. For the purposes of finite computation and approximation, one must work with more 
general spaces, e.g. piece-wise linear, piece-wise polynomial, or piece-wise smooth forms.
As we now describe, there are particular difficulties with defining all of the operators $d, \wedge$, and $\star$ on these 
spaces or their generalizations.

A particularly nice class of differential forms, sufficient for topology, are the \emph{Whitney flat forms}
\cite{Wh}, \cite{DPS}. Roughly speaking, these are $L^\infty$-forms (bounded, measureable, a.e.) 
whose exterior derivative is defined  almost everywhere and is in $L^\infty$. 
A theorem of Wolfe, illuminated in Whitney's book \cite{Wh}, characterizes these forms as precisely the 
bounded linear functionals on chains with respect to a certain norm on chains. They are in fact defined for any Lipschitz manifold.
These differential forms constitute a differential graded
 Banach algebra but, unfortunately, the Hodge star operator does not preserve this space.
 For example, the Hodge star operator does not preserve piece-wise linear, piece-wise polynomial, or piece-wise smooth forms whose singularities occur along the faces of a triangulation.

On the other hand, the space of $L^2$-forms does have a well defined and bounded Hodge star operator, and has the advantage of being a Hilbert space. We may even consider
the space of $L^2$-forms with $d$ in $L^2$, or the Sobolev space of forms with $d$ and $d^*$ in
$L^2$. Unfortunately, the wedge product of forms is not well defined on any of these spaces.

There is also the class of forms on an $n$-manifold given by $k$-forms in $L^{n/k}$ whose
exterior derivative is in $L^{n/k+1}$. By H\"{o}lder's inequality, these form a differential 
graded Banach algebra, but again, the Hodge star operator is not well defined\footnote{It seems useful to have a conceptual explanation for these difficulties.}.

These apparent difficulties are addressed here by considering the space of $L^2$ forms as a module over the
algebra of $L^\infty$-forms.
Specifically, on a compact manifold, there is a continuous inclusion of Whitney flat forms into the Hilbert space of $L^2$-forms. 
If $\omega$ is a Whitney flat form we can define $T(\omega) \in L^2$ by
\begin{align}
\langle T(\omega) , \eta \rangle = \langle d \omega , \omega \wedge \eta \rangle 
\quad \quad \textrm{for all} \quad \eta \in L^2 \label{T}
\end{align}
where $\langle , \rangle$ is the inner product on the Hilbert space of $L_2$-forms.
Indeed, $\omega \wedge \eta \in L^2$ and $d\omega \in L^\infty \subset L^2$, and the right
hand side defines for each $\omega$ a bounded linear functional of $\eta$ on $L^2$ (of norm at most a constant times
$\|d \omega \|_2 \| \omega \|_{\infty}$). As we now explain, if $\omega$ happens to be a smooth form, this vector $T(\omega)$ is the time derivative of the Euler flow starting from $\omega$. 

Let $L_\omega$ denote the bounded linear map from $L^2$ 1-forms given by left multiplication by the 1-form $\omega \in L^\infty$. One can calculate that the adjoint of $L_\omega$, with respect to the inner product, is given by $-\star L_\omega \star$. Therefore, we have 
\[
\langle -\star( \omega \wedge \star d \omega) , \eta \rangle  = \langle d \omega , \omega \wedge \eta \rangle \quad \quad \textrm{for all} \quad \eta \in L^2
\]

From \eqref{eu1} and \eqref{T} we see that for a smooth time dependent 1-form $\omega$ 
the projection of $T(\omega)$ onto co-closed $1$-forms equals $\frac{ \partial \omega}{ \partial t}$, 
the time derivative of an Eulerian fluid flow, i.e. the Euler equation is given by 
\begin{align}
\frac{ \partial \omega}{ \partial t} = \pi \left( T(\omega) \right) \quad \quad 
d^* \omega = 0 \label{eu2}
\end{align}
where $\pi$ is the projection on the the kernel of $d^*$. Similarly, we can re-express the Navier-Stokes equation, as 
\begin{align}
\frac{ \partial \omega}{ \partial t} = \pi \left( T_\nu (\omega) \right) \quad \quad 
d^* \omega = 0 \label{nsT}
\end{align}
where $T_\nu (\omega)$ is defined by
\begin{align}
\la T_\nu (\omega) ,  \eta \ra = \la d \omega , \omega \wedge \eta \ra - \nu \la d \omega , d \eta \ra
\label{Tnu}
\end{align}

The non-linear operator $T_\nu$ can be regarded as a section of the trivial bundle over Whitney flat forms whose fiber is $L^2$-forms. The restriction of this section to smooth forms is an honest vector field, tangent to smooth forms, while for a general Whitney flat form it is not tangent. Of course, this makes solving the ODE starting from an initial Whitney flat form impossible. Nevertheless, this approach enlarges the space to one where the finite models for fluids of the next section can be formulated and compared by estimates.

\section{Finite models}

From the previous section we are motivated to consider the following structure:  a finite dimensional cochain complex $C^\bullet = \{C^j , \delta \}$ with graded commutative product $\cup$ and positive definite inner product $\la , \ra$ on $C^\bullet$. For any such structure, the inner product induces an isomorphism 
from $C^1$ to $\left( C^1 \right)^*$, the dual of $C^1$, and for any $\nu > 0 $ we can define a non-linear flow on $C^1$ given by 
\begin{align}
\frac{ \partial c}{ \partial t} = \pi \left( T_\nu (c) \right) \quad \quad \label{Ceq}
\end{align}
where $\pi$ is the projection onto the kernel of the adjoint $\delta^*$ of $\delta$, and $T_\nu (c)$ is defined by
\[
\la T_\nu (c) ,  b \ra = \la \delta c , c \cup b \ra - \nu \la \delta c , \delta b \ra
\]
Clearly, if $\delta^* c_0 = 0 $ then the flow starting from $c_0$ remains in the kernel of $\delta^*$.

A first example of this structure, associated to any simplicial complex, will be referred to as the \emph{toy-model}. Let $C^\bullet = C^\bullet(K)$ be the simplicial cochains of a closed simplicial complex $K$, with coboundary operator $\delta$. There is a canonical basis of \emph{elementary cochains} given by those whose value is one on a single simplex and zero elsewhere; in this way we can confuse an elementary cochain with the unique simplex on which it is supported.
There is also a graded commutative non-associative product on $C^\bullet$ described easily in terms of the elementary
cochains $a,b$ as follows: $a \cup b$ is zero unless $a$ and 
$b$ intersect in exactly one vertex and span a $(j+k)$-simplex $c$,
in which case, 
\begin{align}
a \cup b =  \epsilon(a,b) \frac{j!k!}{(j+k+1)!} c  \label{cup}
\end{align}
where
$\epsilon(a,b)$ is determined by 
\[
orientation(a) \cdot orientation(b) = \epsilon(a,b) \cdot orientation(c)
\]
see \cite{JD}. Lastly, there is a positive definite inner product on $C^\bullet$ defined by declaring the elementary cochains to be an orthonormal basis.

A second example of this structure, which we call the \emph{Whitney model}, is obtained from the triangulation of any closed Riemannian manifold. For this we take $C^\bullet$ to be simplicial cochains of the triangulation, $\cup$ to be the same 
product as in the last example, and $\la , \ra$ to be the \emph{Whitney metric} which we now describe.

There is a injective cochain map $W$ from $C^\bullet$ to the space of Whitney flat forms \cite{Wh}.
Indeed, $W$ is defined on an elementary cochain $a$ supported on a single simplex $\sigma = [p_{0},p_{1}, \dots ,p_{j}]$ 
with  barycentric coordinates $\mu_{i}$ corresponding to the $i^{th}$ vertex $p_{i}$ of $\sigma$ by
\[
W(a) = j! \sum_{i=0}^{j} (-1)^{i} \mu_{i} \ d \mu_{0} 
\wedge \dots \wedge 
\widehat{d \mu_{i}}
\wedge \dots \wedge d \mu_{j}.
\]
and then is extended linearly to all of $C^\bullet$. Moreover, $RW = Id$ where $R$ is the integration map from forms to cochains of $K$. We will not indicate the role of $K$ in the maps $W$ or $R$.

We define an inner product on $C^\bullet$, as in \cite{Do}, by the restriction of the $L^2$ inner product to the image of $W$, and denote it by the same notation: 
$\la a , b \ra = \la Wa , Wb \ra$. It is shown in 
\cite{Do} that this is a positive definite inner product. 

The ingredients of this Whitney model have nice approximation properties, as indicated in the following theorem, 
proved in \cite{SW}.

\begin{thm} \label{thm:cup} Let $K$ be a triangulation of a manifold $M$ with mesh $\eta$.
There exist a constant $C$ and positive integer $m$, independent of  $K$ such that
\[
\| W(R\omega_1 \cup R\omega_2) - \omega_1 \wedge \omega_2 \|
\leq
C \cdot \lambda ( \omega_1,\omega_2 ) \cdot \eta
\]
where
\[
\lambda (\omega_1 , \omega_2 ) = \|\omega_1\|_{\infty} \cdot 
\|(Id + \Delta)^m \omega_2 \| +  \|\omega_2\|_{\infty} \cdot
\|(Id + \Delta)^m \omega_1 \| 
\]
for all smooth forms $\omega_1,\omega_2 \in \Omega(M)$, where 
$\| \hspace{1em} \|$, $\| \hspace{1em} \|_\infty$ are the $L^2$ and $L^\infty$ norms, respectively.
\end{thm}

\begin{rmk} \label{rmk:WR}
Applying the previous theorem to the constant function $\omega_2=1$ shows that 
$\|WR \omega - \omega\| \leq C \eta$, a result previously obtained in \cite{Do}. 
\end{rmk}

This theorem makes precise the sense in which, in a sequence of triangulations with mesh converging to zero, the cup product converges to the wedge product. As is always the case, in a sequence of refining triangulations we require that the shapes do not become too thin, i.e. the fullness must be bounded; see for example \cite{Do} for details. We call this a \emph{nice sequence} of subdivisions.

 In light of these convergence results, one might hope to prove a strong approximation property of the Whitney model, like the following. For a solution $\omega(t)$ to the Navier-Stokes equations, we could ask that for any given degree of closeness $\epsilon > 0$, and desired time $S>0$, one can choose a small enough mesh so that the solution $c(t) = (R\omega_0) (t)$ to the Whitney model starting from $R\omega_0$ is within $\epsilon$ of $\omega(t)$ for all 
 $0 < t < S$. This is in fact too strong of a request since, for example, the equation may not be stable near a given smooth initial condition.

Instead we are able to show that  the local operators used to define 
the Navier-Stokes are approximated well, in a weak sense, by the Whitney model. In particular, the Whitney model provides a weak approximation to the time derivative of the 
Navier Stokes flow starting from a smooth solution.

\begin{thm} 
Let $\omega$ be a smooth form on a closed Riemannian manifold $M$ and $K_1 , K_2, \ldots$ be a nice sequence of triangulations of $M$. Let $c = R \omega$ for each triangulation.

Let $\pi \left( T_\nu \omega \right)$
and  $\pi \left(T_\nu c \right)$ be the tangent vectors defining the Navier-Stokes flow on $M$ and the Whitney models, respectively, defined as above. 
Then $W \pi (T_\nu c)$ converges weakly to $\pi ( T_\nu \omega)$ in the sense that, for all $\eta \in L^2$, 
\[
\la W \pi (T_\nu c), \eta \ra \to \la \pi ( T_\nu \omega) , \eta \ra 
\]
as the mesh tends to zero.\end{thm}

\begin{proof} We consider the two summands of \eqref{Tnu} separately.
For all $\omega$ and $\eta$ we have that 
\[
\la W \delta R \omega , \eta \ra = \la WR d \omega , \eta \ra \to \la d \omega , \eta \ra
\]
by Remark~\ref{rmk:WR}. For the non-linear term we denote $T_0$ by $T$, and have that
\begin{align*}
\lim \la W \pi T R \omega , \eta \ra = \lim \la W \pi T R \omega , WR \eta \ra = \lim \la T R \omega , \pi R \eta \ra
\end{align*}
by Remark~\ref{rmk:WR} and the fact that orthogonal projection is self adjoint. It is known that
the Hodge decomposition of cochains, with respect to the Whitney metric, converges to the smooth Hodge decomposition (see
\cite{DP} Theorem 2.10, and \cite{Da} Lemma 3.18 for a recent improved estimate).
Therefore, we have that $\lim \pi R \eta = \lim R \pi \eta$, and that the previous expression is equal to 
\begin{align*}
\lim \la T R \omega , R \pi \eta \ra &= \lim \la \delta R \omega , R \omega \cup R \pi \eta \ra = \lim \la W \delta R \omega , W\left( R \omega \cup R \pi \eta \right)  \ra
\end{align*}
Now we use that the cup product converges to the wedge product, Theorem~\ref{thm:cup}, so the previous expression
is equal to
\begin{align*}
\la d \omega , \omega \wedge \pi \eta \ra = \la T \omega , \pi \eta \ra = \la \pi T \omega , \eta \ra ,
\end{align*}
and this completes the proof.

We remark that the rate of convergence is on the order of mesh size constant factor depending only on
the norm of $\omega$, $\eta$, their derivatives, and some universal constants independent of the triangulations.
\end{proof}


There is a purely algebraic viewpoint of the argument above. Note that the operator 
$\star L_\omega \star$ corresponds $d^*$-homologically to the cap product of $\omega$ on the complex
of forms (with differential $d^*$). This defines a differential module of  $\left( \Omega , d^* \right)$
over $\left( \Omega(M), d , \wedge \right)$. By the argument above, the analogous capping operation on $C^\bullet$, defined as the adjoint of cup product, converges weakly to this module structure on forms. One can of course state the flow \eqref{Ceq} purely in terms of this cap product, though we have not chosen to do so.
 
From the previous result we also obtain the following result concerning steady state solutions.

\begin{cor} \label{cor:steady}
Let $M$ be a closed Riemannian manifold. The following is a necessary condition for a smooth 1-form 
$\omega$ to be a steady state solution to the Navier-Stokes equation:
for every nice sequence of subdivisions of a given triangulation of $M$ the sequence of cochains $c=R\omega$ given by integrating $\omega$ over each subdivision must converge weakly to a steady state solution of the Whitney model as the mesh converges to zero.
\end{cor}

\begin{proof}
Steady state solutions are those forms $\omega$ which satisfy $\pi \left( T_\nu (\omega) \right)=0$.
By the previous proposition, this is approximated weakly by $W \pi (T_\nu c)$.
\end{proof}

We remark that for the Whitney model associated to a particular triangulation, the integral of a smooth steady state solution may not
be a steady state solution in the Whitney model. Also, at any finite stage, the Whitney model
may have steady state solutions that are not given by the integral of smooth steady state solutions. 
For example, it is clear from \eqref{Ceq} that a cochain $c$ in the Whitney model is a steady state solution to Euler's equation if the ``vorticity''
$\delta c \in C^2$ is orthogonal to the image of the operator $L_c : C^1 \to C^2$
given by cupping with $c$. This condition is intimately related to not only the metric, but also to the combinatorics of the triangulation.

So far we have considered only piece-wise linear differential forms for our finite models. It is clear that using higher degree polynomial forms would yield better approximation results. For example, the recent finite element spaces in \cite{Da} may be used in a similar way. 

\section{Remarks on finite models}

The finite models of the previous section, consisting of a finite cochain complex $C^\bullet = \{C^j , \delta \}$ with graded commutative product $\cup$ and positive 
definite inner product $\la , \ra$, also enjoy the properties derived in section 1. 
In particular, the harmonic elements, i.e. those satisfying $\delta c = \delta^* c = 0$, are steady-state solutions to \eqref{Ceq}. Also, the same proof as in \eqref{norm}, now using the graded commutativity of $\cup$, shows that 
\begin{align} \label{norm2}
\frac{ \partial}{ \partial t} \| c \|^2 = -2 \nu \| \delta c \|^2
\end{align}
holds for any co-closed solution to \eqref{Ceq}.

Since the finite models are defined by algebraic equations on a finite dimensional vector space, the flow for $\eqref{Ceq}$ is defined for all time, until blow up. From \eqref{norm2} we see that any finite model Navier-Stokes or Euler flow is 
in fact defined for all time since the norm is non-increasing.

We remark that the toy-model and Whitney model may be implemented computationally by expressing the
coboundary operator $\delta$, the cup product $\cup$ and the inner product $\la , \ra$ in terms of the basis of elementary cochains. There is a local calculation that one must perform to determine the Whitney metric from the given Riemannian metric. The last required operator is the orthogonal projection onto the kernel of $\delta^*$, which may be computed for example by first computing the Hodge decomposition.

It is worth noting that the Whitney metric and cup product \eqref{cup} are semi-local with respect to the basis of elementary cochains: two elementary cochains have non-zero (inner) product only if the they are supported on simplices which are faces of a common top dimensional simplex. The inner product of the toy-model is even more local, as the
elementary cochains are an orthonormal basis. 
On the other hand, the orthogonal projection operator in \eqref{Ceq} is non-local.

We do not presently have a means to compare the two finite models we've given, though it is conceivable that the toy model alone will produce realistic fluid models. To make this
precise, one needs a notion of a (potentially soft) map or morphism between finite models or sequences of them. In
 \cite{DS} a related issue is addressed for the algebraic structure of Euler's equation.  We hope the pursuit of similar  algebraic ideas along with the formulations of this paper will lead to a further understanding of fluids and fluid models.

\bigskip

{\sc Scott O. Wilson; 303 Kiely Hall; 65-30 Kissena Blvd; Flushing, NY 11367 USA.}

email: {\tt scott.wilson@qc.cuny.edu}

\end{document}